\documentclass[12pt]{amsart}
\usepackage{a4wide, amssymb, bbm, calligra, mathrsfs}
\usepackage[2cell,arrow,curve,matrix]{xy}
\UseHalfTwocells
\UseTwocells

\def\A{{\mathbb A}}
\def\B{{\mathbb B}}

\def\H{{\mathbb H}}
\def\I{{\mathbb I}}

\def\M{{\mathbb M}}

\def\P{{\mathbb P}}
\def\Q{{\mathbb Q}}
\def\R{{\mathbb R}}
\def\S{{\mathbb S}}

\def\Z{{\mathbb Z}}

\def\Ho{{\bf Ho}}

\let\al\alpha
\let\bb\beta

\def \ext{\mathop{\sf Ext}\nolimits}

\newtheorem{Pro}{Proposition}
\newtheorem{Le}[Pro]{Lemma}
\newtheorem{The}[Pro]{Theorem}
\newtheorem{Co}[Pro]{Corollary}
\theoremstyle{definition}
\newtheorem{De}[Pro]{Definition}
\theoremstyle{remark}

\newtheorem{Rem}[Pro]{Remark}

\def\al{{\alpha}}
\def\bb{{\beta}}

\def\Ty{{\frak T}{\frak y}{\frak p}{\frak e}{\frak s}}

 \def\ta{{\frak T}}

\def\xyma{\xymatrix@M.7em}

\begin{document}

\title{ Projective and injective objects in symmetric categorical groups}

\author[T. Pirashvili]{Teimuraz  Pirashvili}
\address{
Department of Mathematics\\
University of Leicester\\
University Road\\
Leicester\\
LE1 7RH, UK} \email{tp59-at-le.ac.uk}
\thanks{Research was partially supported by the GNSF Grant
ST08/3-387}

\maketitle

\hfill{\emph{ Dedicated to the memory  of Prof. V. K. Bentkus}}

\

Categorical rings (called also 2-rings)  were introduced in
\cite{cat_rings}.  Categorical modules (called also 2-modules) over
a categorical ring  were introduced in \cite{vincent}. Categorical
modules and symmetric categorical groups \cite{ext_sym_cat_gr} are
examples of  abelian 2-categories studied in \cite{dupont} (for
other examples of abelian 2-categories see \cite{ab-2-ab}).
Projective objects in the framework of symmetric categorical groups
 were introduced in \cite{ext_sym_cat_gr} have an obvious generalization
 to the case of abelian 2-categories. Hence by duality one can also talk on
 injective objects in any abelian 2-category.
 Moreover in \cite{ext_sym_cat_gr} the authors
conjectured that the abelian 2-category of symmetric categorical
groups have enough projective objects. In the course of our work
\cite{2-robi} we noted that  the 2-category of symmetric categorical
groups have enough projective and injective objects. Moreover this
statement is a trivial consequences of the cohomological description
of abelian 2-category of symmetric categorical groups obtained first
in \cite{sinh}. Using base change argument this result also yields
that the 2-category of categorical modules over any categorical ring
have enough projective and injective objects.

Quite recently these results were announced in \cite{china} but with
wrong proofs (Lemmata 3 and 11 of \emph{loc. cit.} are both wrongs).
Here we give our original proofs.

A groupoid enrich category  $\ta$  is a 2-category such that any
2-arrow is invertible. If $\ta$ is a groupoid enrich category  then
we use the word ''morphism'' for  $1$-morphisms and we use the word
''track'' for 2-morphisms.   We let $\Ho(\ta)$ be the corresponding
homotopy category. If $f,g:A\to B$ are morphism in $\ta$, then we
say that $f$ is homotopic to $g$ if there exists a track from $f$ to
$g$. Let $\Ho(\ta)$ be the corresponding homotopy category. Thus
 objects of $\Ho(\ta)$ are the same as of
$\ta$, while morphisms in $\Ho(\ta)$ are homotopic classes of
morphisms in $\ta$.

Symmetric categorical groups
 form a groupoid enrich category which is denoted by ${\bf
SymCatGr}$. In particular one can form the homotopy category of the
2-category of symmetric categorical groups. We will denote the
corresponding homotopy category by $\Ho$. This category has the
following nice description which follows from the classical results
of H.X.Sinh  \cite{sinh}.

First we fix some notations. If $\S$ is a symmetric categorical
group then we let $\pi_0(\S)$ and $\pi_1(\S)$ denote respectively
the abelian group of connected components of $\S$ and the abelian
group of all automorphisms of the neutral object of $\S$. For
symmetric categorical groups $\S_1$ and $\S_2$ we have a groupoid
(in fact a symmetric categorical group \cite{ext_sym_cat_gr}) ${\bf
Hom}(\S_1,\S_2)$. To describe $\pi_i({\bf Hom}(\S_1,\S_2))$ ,
$i=0,1$ we need to introduce the  category $\Ty$. Objects of the
category $\Ty$ are triples $\A=(A_0,A_1,\al)$,  where $A_i$ is an
abelian group, $i=0,1$ and
$$\al\in \hom(A_0/2A_0,A_1)=\hom(A_1,\, _2A_1)$$
Here for an abelian group $A$ we set $$_2A=\{a\in A| 2a=0\}$$
 A morphism $f$ from $\A=(A_0,A_1,\al)$ to $\B=(B_0,B_1,\bb)$ in $\Ty$ is
given by a pair $f=(f_0,f_1)$, where $f_0:A_0\to B_0$ and
$f_1:A_1\to B_1$ are homomorphisms of abelian groups such that $\bb
f_0 = f_1\al$.

Let $\S$ be a symmetric categorical group. We put
$$type(\S):=(\pi_0(\S),\pi_1(\S),k_\S)$$
where $k_\S$ is the homomorphism induced by the commutativity
constrants in $\S$. Both categories $\Ho$ and $\Ty$ are additive and
the functor
$$type:\Ho\to \Ty$$
is additive.  According to \cite{sinh} the functor $type$ is full,
essentially surjective on objects and the kernel of $type$
(morphisms which goes to zero) is a square zero ideal of $\Ho$. It
follows then that the functor $type$ reflects isomorphisms and
induces a bijection on the isomorphism classes of objects. More
precisely, for any symmetric categorical groups $\S_1$ and $\S_2$
one has a short exact sequence of abelian groups
\begin{equation}\label{sinh0}
0\to \ext(\pi_0(\S_1),\pi_1(\S_2))\to \pi_0({\bf Hom}(\S_1,\S_2))\to
\Ty(type(\S_1),type(\S_2))\to 0
\end{equation} Furthermore one has also an isomorphism of abelian
groups
\begin{equation}\label{sinh1}
\pi_1({\bf Hom}(\S_1,\S_2))\cong
\hom(\pi_0(\S_1),\pi_1(\S_1))\end{equation}
 These facts greatly simplifies to work with symmetric
categorical groups.

For a given object $\A$ of the category $\Ty$ we choose a symmetric
categorical group $H(\A)$ such that $type(H(\A))=\A$. Such object
exist and is unique up to equivalence. Moreover, for any morphism
$f:\A\to \B$ we  choose a morphism of symmetric categorical groups
$H(f):H(\A)\to H(\B)$, such that $type(H(f)=f$. The reader must be
aware that the assignments $\A\to H(\A)$, $f\mapsto H(f)$ does NOT
define a functor $\Ty\to \Ho$.

Recall that \cite{ext_sym_cat_gr} a morphism $F:\S_1\to \S_2$ of
symmetric categorical groups  is called \emph{essentially
surjective} (resp. \emph{faithful}) if it is epimorphism on $\pi_0$
(resp. monomorphism on $\pi_1$). A symmetric categorical group $\S$
is called \emph{projective}  provided for any essentially surjective
functor $F:\S_1\to \S_2$ and a morphism $G:\S\to \S_2$ there exist a
morphism $L:\S\to \S_1$ and a track from $FL\to G$.  Dually a
symmetric categorical group $\S$ is called \emph{injective} provided
for any faithful functor $F:\S_1\to \S_2$ and a morphism $G:\S_1\to
\S$ there exist a morphism $L:\S_2\to \S$ and a track from $LF$ to
$G$.

We can develop same sort of language in the category $\Ty$.
 A morphism $f=(f_0,f_1)$ in $\Ty$
is \emph{essentially surjective} if $f_0$ is epimorphism of abelian
groups. Moreover an object $\P$ in $\Ty$ is \emph{projective} of for
any essentially surjective morphism $f:\A\to \B$ in $\Ty$ the
induced map
$$\Ty(\P,\A)\to \Ty(\P,\B)$$ is surjective.

Dually, a morphism $f$ in $\Ty$ is \emph{faithful} provided $f_1$ is
injective and an object $\I=(I_0,I_1,\iota)$ of $\Ty$ is
\emph{injective} if for any faithful morphism $f:\A\to \B$ in $\Ty$
the induced map
$$\Ty(\B,\I)\to \Ty(\A,\I)$$ is surjective.

 It is clear that a morphism $F:\S_1\to \S_2$
of symmetric categorical groups is \emph{essentially surjective}
(resp. \emph{faithful}) iff $type(F): type(\S_1)\to type(\S_2)$ is
so in $\Ty$.

For an abelian group $M$ we introduce two objects in $\Ty$:
$$l(M):=(M,M/2M,id_{M/2M})$$
$$r(M)=(_2M,M,id_{_2M})$$

\begin{Le}\label{adj}  i) If $M$ is an abelian group and
$\A=(A_0,A_1,\al)$ is an object in $\Ty$, then
 one has following functorial isomorphisms of abelian groups
$$\Ty(l(M),\A)=\hom(M,A_0),$$
$$\Ty(\A,r(M))=\hom(A_1,M).$$

 ii) For any free abelian group $P$ the object $l(P)\in \Ty$ is
projective in $\Ty$, dually for any divisible abelian group $Q$ the
object $r(Q)\in \Ty$ is injective.

iii) For any free abelian group $P$ the symmetric categorical group
$H(l(P))$ is projective symmetric categorical group and dually for
any divisible abelian group $Q$ the triple $r(Q)$ is injective.
\end{Le}
\begin{proof} i) and ii) are obvious. Let $F:\S_1\to \S_2$ be an
 essentially surjective morphism of symmetric categorical groups and
 $G:H(l(P))\to \S_2$ be a morphism of symmetric categorical groups.
 Apply the functor $type$ to get a essentially surjective morphism
 $type(F):type(\S_1)\to type(\S_2)$ in $\Ty$ and a morphism
$type(G):l(P)\to type(\S_2)$ in $\Ty$. Since
$\pi_0(F):\pi_0(\S_1)\to\pi_0(\S_2)$ is an epimorphism of abelian
groups, $F$ is a free abelian group the homomorphism $\pi_0(G):P\to
\pi_0(\S_2)$ has a lifting to the homomorphism $P\to \pi_0(\S_1)$
 Since $P$ is free abelian it follows from the exact sequence
(\ref{sinh0}) that  for $i=0,1$ one has an isomorphism
\begin{equation}\label{sinh2}\pi_0({\bf Hom}(H(l(P)),\S_i))\cong \Ty(l(P),type(\S_i))\cong
\hom(P,\pi_0(\S_i))\end{equation} Take a morphism $L:H(l(P)\to \S_1$
of symmetric categorical groups  which corresponds to the
homomorphism $P\to \pi_0(\S_1)$. By our construction one has an
equality $type(FL)=type(G)$, which imply that the class of $FL$ and
of $G$ in $\pi_0({\bf Hom}(H(l(P)),\S_1))$ are the same. Thus there
exist a track from $FL$ to $G$. This shows that $H(l(P))$ is a
projective symmetric categorical group. a dual argument works for
injective objects.

\end{proof}

\begin{Pro}

The 2-category of symmetric categorical groups have enough
projective and injective objects.

\end{Pro}

\begin{proof} Let $\S$ be a symmetric categorical group. Choose a
free abelian group $P$ and epimorphism of abelian groups $f_0:P\to
\pi_0(\S)$. By Lemma \ref{adj} it has a unique extension to a
morphism $f=(f_0,f_1):l(P)\to type(\S)$ which is essentially
surjective. Since $P$ is  is a free abelian group, we have the
isomorphism (\ref{sinh2}), which show that   there exist a morphism
of symmetric categorical groups $H(l(P)) \to \S$ which realizes
$f_0$ on the level of $\pi_0$. Clearly this morphism does the job.

Dually, choose a monomorphism $g_1:\pi_1(\S) _1\to Q$ with divisible
abelian group $Q$. By Lemma \ref{adj} it has the unique extension as
a morphism $g:type(\S)\to r(Q)$ which is faithful by the
construction. Since $\pi_1(r(Q))=Q$ is injective object in the
category of abelian groups by the short exact sequence (\ref{sinh0})
we have
$$\pi_0({\rm Hom}(\S, H(r(Q))))\cong \Ty(type(\S),r(Q))\cong
\hom(\pi_1(\S),Q)$$ which shows that $g$ can be realized as a
morphism of symmetric categorical groups and we get the result.
\end{proof}

\begin{Pro} Let $\R$ be a categorical group. Then the category of
categorical right $\R$-modules have enough projective and injective
objects.
\end{Pro}

\begin{proof} By Yoneda Lemma for symmetric categorical groups the categorical ring $\R$
considered as a right $\R$-module is projective and from this fact
one easily deduces the statement on projective objects. For
injectivity we consider the 2-functor ${\bf Hom}(\R,-)$ from the
2-category of symmetric categorical groups to the 2-category of
categorical right $R$-modules. It is a right 2-adjoint to the
forgetful 2-functor. Since the forgetful functor is exact it follows
that the 2-functor ${\bf Hom}(\R,-)$ takes injective objects to
injective ones. Let $\M$ be a categorical left $\R$-module. Choose a
faithful morphism $\M\to \Q$ in the 2-category of symmetric
categorical groups with injective symmetric categorical group $\Q$.
Apply now the 2-functor ${\bf Hom}(\R,-)$. It follows from the
isomorphism (\ref{sinh1}) that  ${\bf Hom}(\R,\M)\to  {\bf
Hom}(\R,\Q)$ is a faithful morphism of right $\R$-modules. By the
same reasons the obvious morphism $\M\to {\bf Hom}(\R,\M)$ is also
faithful. Taking the composite we obtain a faithful morphism $\M\to
{\bf Hom}(\R,\Q)$ and hence the result.
\end{proof}

Note that the proof of the last statement is essentially the same as
it was for classical rings. The same is also true for the following
result and because of this we omit the proof. Recall that if $\ta$
is an additive 2-category and $M$ is an object in $\ta$ then one has
the categorical ring ${\bf Hom}(M)$ (compare
\cite{dupont},\cite{vincent}).

\begin{Pro} If $\ta$ is a 2-abelian category which posses a small
projective generator $P$, then $\ta$ is 2-equivalent to the category
of right categorical modules over the categorical ring  ${\bf
Hom}(P,P)$.
\end{Pro}

Consider the following symmetric categorical group $\H$. Objects of
the groupoid $\H$ are integers. If $n\not =m$ then there is no
morphism from $n$ to $m$, $n,m\in \Z$. The group of automorphisms of
$n$ is the cyclic group of order two with generator $\epsilon_n$,
$n\in \Z$. The monoidal functor is induced by the addition of
integers. The associativity and unite constrants are identity
morphisms, while the commutativity constrant $n+m\to m+n$ equals to
$\epsilon_{n+m}$. By our construction $\H$ is a small projective
generator in the 2-category of symmetric categorical groups. Hence
we obtained the following important fact.

\begin{Pro} The 2-category of symmetric categorical groups is
2-equivalent to the category of right categorical modules over the
categorical ring ${\bf Hom}(\H,\H)$.
\end{Pro}

\end{document}